\documentclass[11pt]{article}
\usepackage{graphicx} 
\usepackage{geometry}
\usepackage{pgfplots}
\pgfplotsset{compat=1.18}

\usepackage{graphicx} 
\usepackage{tikz}
\usetikzlibrary{patterns}
\usepackage{float} 
\usetikzlibrary{shapes}
\usetikzlibrary{decorations.pathreplacing}
\usetikzlibrary{decorations.pathmorphing}
\usetikzlibrary{positioning}
\usepackage{xcolor}
\usepackage{appendix}
\usepackage{enumitem}
\usepackage{natbib}
\usepackage{booktabs}
\usepackage{amsthm} 
\usepackage{amsmath,bm}
\usepackage{amssymb}
\usepackage{hyperref}
\usepackage{verbatim}

\usepackage[capitalise]{cleveref}

\newtheorem{theorem}{Theorem}[section]
\newtheorem{lemma}[theorem]{Lemma}
\newtheorem*{remark}{Remark}

\newtheorem{claim}[theorem]{Claim}

\newtheorem{proposition}[theorem]{Proposition}

\newtheorem{construction}[theorem]{Construction}

\newenvironment{poc}{\begin{proof}[Proof of claim]}{\end{proof}}
\newenvironment{mycase}[1]{%
    \par\addvspace{\medskipamount}%
    \noindent\textbf{#1}\quad%
}{%
    \par\addvspace{\medskipamount}%
}

\title{Edge density and minimum degree thresholds for $H$-free graphs with unbounded chromatic number}
\author{
Zhuo Wu\thanks{Departament de Matemàtiques, Universitat Politècnica de Catalunya (UPC),
Carrer de Pau Gargallo 14, 08028 Barcelona, Spain. Z. Wu acknowledges the bilateral AEI+DFG research project PCI2024-155080-2: SRC-ExCo – Structure, Randomness and Computational Methods in Extremal Combinatorics, and the PID2023-147202NB-I00 (COCOA: COntemporary COmbinatorics and its Applications), 
all funded by MICIU/AEI/10.13039/501100011033. Email: \texttt{zhuo.wu@upc.edu}}
\and
Yisai Xue\thanks{\textit{Corresponding author}. School of Mathematics and Statistics, Ningbo University, Ningbo, China. Supported by the National Natural Science Foundation of China (No. 12501486). Email: \texttt{xueyisai@nbu.edu.cn}}}

\date{}

\begin{document}

\maketitle

\begin{abstract}
The chromatic threshold $\delta_\chi(H)$ of a graph $H$ is the infimum of $d>0$ such that the chromatic number of every $n$-vertex $H$-free graph with minimum degree at least $dn$ is bounded in terms of $H$ and $d$. A breakthrough result of Allen, Böttcher, Griffiths, Kohayakawa, and Morris determined $\delta_\chi(H)$ for every graph $H$; in particular, if $\chi(H)=r\ge 3$, then $\delta_\chi(H) \in\{\frac{r-3}{r-2},~\frac{2 r-5}{2 r-3},~\frac{r-2}{r-1}\}$. 

In this paper we investigate the trade-off between minimum degree and edge density in the critical window around the chromatic threshold. 
For a fixed graph $H$ with $\chi(H)=r$,  allowing a constant deficit below $\delta_\chi(H)$, 
we prove sharp (up to lower-order terms) upper bounds on the edge density of $n$-vertex $H$-free graphs whose chromatic number diverges. 
Equivalently, within this degree regime we show that a suitable global bound on the number of edges forces the chromatic number to remain bounded. 
Our results thus quantify how global edge density can compensate for a deficit in the local minimum-degree condition near $\delta_\chi(H)$; more specifically, we obtain explicit bounds in two of the three possible cases arising in the trichotomy of $\delta_\chi(H)$. Our extremal constructions---based on Erdős graphs and blowups of Borsuk--Hajnal graphs---show that these bounds are best possible up to $o(n^2)$ terms.
\end{abstract}

\section{Introduction}

A central problem in extremal graph theory concerns the relationship between local density constraints and global structural properties. 
The \textbf{chromatic number} of a graph $G$, denoted $\chi(G)$, is the smallest integer $k$ such that the vertex set $V(G)$ can be partitioned into $k$ independent sets.
A natural and well-studied question is how local conditions---such as large minimum degree together with the exclusion of a fixed subgraph $H$---restrict the possible chromatic number of~$G$.

This line of inquiry leads to the concept of the \textbf{chromatic threshold} $\delta_\chi(H)$, which represents the critical minimum degree above which any $H$-free graph must have a bounded chromatic number. Formally,
\begin{align*}
\delta_\chi(H):=
& \inf \{d: \exists~C=C(H, d)\text{ such that if $G$ is a graph on $n$ vertices}, \\
& \text{with $\delta(G) \geq d n$ and $H \not \subseteq G$, then $\chi(G) \leq C$}\}.
\end{align*}
Progress toward understanding chromatic thresholds was made initially for several special graph classes, including cliques~\cite{goddard2011dense,nikiforov2010chromatic,thomassen2002chromatic}, odd cycles~\cite{thomassen2007chromatic}, and near-bipartite graphs~\cite{luczak2010coloring}. 
These developments paved the way for the breakthrough theorem of Allen, B{\"o}ttcher, Griffiths, Kohayakawa, and Morris~\cite{allen2013chromatic}, who resolved the problem in full by determining $\delta_\chi(H)$ for every graph~$H$.  
Their characterization yields the following trichotomy: if $\chi(H)=r\ge3$, then
\[\delta_\chi(H)\in\left\{\frac{r-3}{r-2},\;\frac{2r-5}{2r-3},\;\frac{r-2}{r-1}\right\}.\]
Moreover, $\delta_\chi(H) \neq \frac{r-2}{r-1}$ if and only if $H$ has a forest in its decomposition family, and $\delta_\chi(H)= \frac{r-3}{r-2}$ if and only if $H$ is $r$-near-acyclic. (For the precise definitions of the decomposition family and the concept of being $r$-near-acyclic, see \cref{sec:pre}.)
For further recent progress, see~\cite{bourneuf2025denseneighborhoodlemmaapplications,huang2025interpolating,liu2024beyond,xue2025directed}.

Recently, Kim, Liu, Shangguan, Wang, Wu, and Xue~\cite{liu-shangguan-wu-xue} proved a sharp stability theorem: any $H$-free graph with minimum degree $\delta(G)\ge (\delta_\chi(H)-o(1))n$ and large chromatic number must be structurally close to an extremal configuration.  This raises a natural question: what structural guarantees remain when the minimum-degree condition is weakened to allow a linear deficit, i.e., when $\delta(G)>(\delta_\chi(H)-c)n$?  In this paper we resolve this question by proving a new type of stability result. We show that stability can be recovered under such a degree deficit provided that a suitable global constraint is imposed on the total number of edges.  More precisely, our main theorem establishes a sharp upper bound on the number of edges in an $H$-free graph whose chromatic number is unbounded under the relaxed minimum-degree condition.  This demonstrates that sufficient global density---measured by the edge count---can indeed compensate for a linear deficit in the local minimum-degree assumption, thereby forcing the chromatic number to remain bounded.

Our main theorems give a precise quantitative expression of this principle. The resulting bounds depend sensitively on the value of the chromatic threshold $\delta_\chi(H)$, and we obtain sharp results in two of the three cases appearing in the trichotomy of Allen et~al.

We remark that the third case of the trichotomy, where $\delta_\chi(H) = \frac{r-2}{r-1}$, is trivial in the context of our study. In this regime, the threshold coincides with the Tur\'an density of $H$, it is impossible to increase the global edge density to compensate for a deficit in the minimum degree. 
Consequently, we focus our attention on the remaining two cases.

\begin{theorem}\label{thm:thm1}
  Let $H$ be a graph with $\chi(H)=r\ge 3$ and $\delta_\chi(H)=\tfrac{2r-5}{2r-3}$. 
  For $\tfrac{2r-5}{2r-2}\le \delta\le \tfrac{2r-5}{2r-3}$, every $n$-vertex $H$-free graph $G$ with minimum degree $\delta(G)\ge\delta n$ and $\chi(G)=\omega(1)$ satisfies
  \[
  e(G)\le \bigl(f_1(r,\delta)+o(1)\bigr)n^2,
  \]
  where \[
f_1(r,\delta) =
\begin{cases}
\dfrac{(-5r^2+17r-14)\delta^2+(10r^2-44r+46)\delta-5r^2+27r-36}{2},
& \text{if } \tfrac{5r-13}{5r-8}<\delta\le\tfrac{2r-5}{2r-3}, \\[1.2ex]
\dfrac{-4 (r-1)\delta^2+4 (2 r-5)\delta+r-3}{10r-28},
& \text{if }\tfrac{2r-5}{2r-2}\le\delta\le\tfrac{5r-13}{5r-8}.
\end{cases}
\]
\end{theorem}

To illustrate the behavior of this bound, we plot $f_1(r,\delta)$ for the case $r=4$ in Figure~\ref{fig:f1_r4}, which highlights the transition between the two regimes at $\delta = 7/12$.

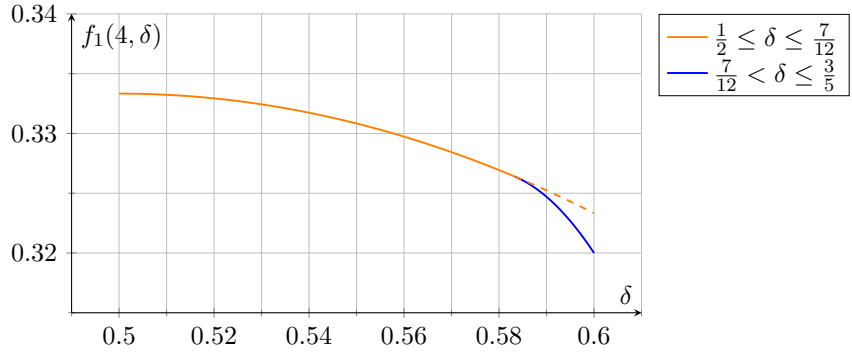
\begin{figure}
    \centering
    \begin{tikzpicture}[scale=0.9]
    \begin{axis}[
        axis lines=middle,
        xlabel={$\delta$},
        ylabel={$f_1(4,\delta)$},
        xmin=0.5, xmax=0.6,
        ymin=0.325, ymax=0.33,
        grid=both,
        minor tick num=1,
        samples=200,
        restrict y to domain=-1:1,
        legend pos=outer north east,
        width=10cm,
        height=6cm,
        enlargelimits={abs=0.01},
    ]
        \addplot[orange, thick, smooth, domain=0.5:0.5833] {-x^2 + x + 1/12};
        \addplot[blue, thick, smooth, domain=0.5833:0.6] {-13*x^2 + 15*x - 4};
        \addplot[orange, thick, dashed, smooth, domain=0.5833:0.6] {-x^2 + x + 1/12};
        \legend{{$\frac{1}{2}\le\delta\le\frac{7}{12}$}, {$\frac{7}{12}<\delta\le\frac{3}{5}$}};
    \end{axis}
\end{tikzpicture}
    \caption{The edge density bound function $f_1(4,\delta)$ within the critical window.}
    \label{fig:f1_r4}
\end{figure}

\begin{remark}
The bound in Theorem~\ref{thm:thm1} is asymptotically tight; explicit constructions based on blowups of Borsuk–Hajnal graphs matching these densities are provided in \cref{sec2.1} (see Figure~\ref{fig:construction} (a) and (b)).
\cref{thm:thm1} can be stated as a sufficient condition for bounded chromatic number in terms of a ``density-deficit" trade-off. 
Let $\rho = e(G)/n^2$ be the edge density. 
Theorem~\ref{thm:thm1} implies that if
\[
h(\rho, \delta) := \rho + A(\delta - B)^2 > C,
\]
then $\chi(G)$ is bounded. 
Notably, this trade-off is linear in the global density $\rho$ but quadratic in the local parameter $\delta$.
The parameters $A, B, C$ depend on the range of $\delta$:

\begin{itemize}
    \item \textbf{Upper Regime} ($\frac{5r-13}{5r-8} < \delta \le \frac{2r-5}{2r-3}$):
    In this range, the parameters are given by:
    \[
    A = \frac{5r^2-17r+14}{2}, \quad 
    B = \frac{5r^2-22r+23}{5r^2-17r+14}, \quad 
    C = \frac{5r^2-22r+25}{2(5r^2-17r+14)}.
    \]
    
    \item \textbf{Lower Regime} ($\frac{2r-5}{2r-2} \le \delta \le \frac{5r-13}{5r-8}$):
    In this range, the parameters are given by:
    \[
    A = \frac{2r-2}{5r-14}, \quad 
    B = \frac{2r-5}{2r-2}, \quad 
    C = \frac{r-2}{2r-2}.
    \]
\end{itemize}
\end{remark}

\begin{theorem}\label{thm:thm2}
  Let $H$ be a graph with $\chi(H)=r\ge 4$ and $\delta_\chi(H)=\tfrac{r-3}{r-2}$. 
  For $\tfrac{r-3}{r-1}\le \delta\le \tfrac{r-3}{r-2}$, every $n$-vertex $H$-free graph $G$ with minimum degree $\delta(G)\ge\delta n$ and $\chi(G)=\omega(1)$ satisfies
  \[
  e(G)\le \bigl(f_2(r,\delta)+o(1)\bigr)n^2,
  \]
  where \[
f_2(r,\delta)=\delta(1-\delta)+\frac{(1-\delta)^2}{4}+\frac{(r-4)\delta^2}{2(r-3)}.
\]
\end{theorem}

The bound presented in \cref{thm:thm2} are sharp, as demonstrated by explicit constructions based on Erd\H{o}s graphs (see  Figure \ref{fig:construction} (c)). 

\begin{remark}
The lower bounds on~$\delta$ in our theorems are genuine.  At the lower endpoint of each range, our bounds coincide with the classical Turán bound for $K_r$-free graphs, and the extremal configurations reduce to complete $(r-1)$-partite graphs.  Below this threshold no stronger estimate is possible, since Turán extremal graphs already maximise the number of edges under mere $H$-freeness.  
Thus our results capture exactly the range in which one can obtain edge-density bounds that go strictly beyond the Turán extremal behaviour.
\end{remark}

\section{Preliminaries}\label{sec:pre}

\paragraph{Notation.}

For positive integers $k \le n$, let $[n]=\{1,\ldots,n\}$, and let $\binom{[n]}{k}$ denote the family of all $k$-element subsets of $[n]$.  
We use standard asymptotic notation as $n\to\infty$: $o(1)$ denotes a quantity tending to~$0$, while $\omega(1)$ denotes a quantity tending to~$\infty$.  
For clarity of presentation, we omit floors and ceilings whenever they are not essential.

Let $G=(V,E)$ be a graph, where $V$ is the vertex set and $E$ is the edge set.  
The \textbf{order} of $G$ is $|V(G)|$, and the \textbf{number of edges} is $e(G):=|E(G)|$.  
For a vertex $v\in V(G)$, its \textbf{degree} is $d(v)$, and the \textbf{minimum degree} of $G$ is $\delta(G)=\min_{v\in V(G)} d(v)$.  
The \textbf{girth} of $G$, denoted $g(G)$, is the length of its shortest cycle.  
For $U\subseteq V(G)$, the \textbf{induced subgraph} $G[U]$ is the graph with vertex set $U$ and all edges with both endpoints in $U$.  
For $v\in V(G)$, we write $G-v$ for the subgraph induced by $V(G)\setminus\{v\}$.  
For two graphs $G$ and $H$, their \textbf{join} $G\vee H$ is obtained by taking disjoint copies of $G$ and $H$ and adding all edges between $V(G)$ and $V(H)$.

Given a graph $G$, let $\alpha(G)$ denote the size of a largest independent set in $G$.  
More generally, Hajnal, and independently Erd\H{o}s and Rogers \cite{erdos1962construction}, introduced the \textbf{$K_p$-independence number} $\alpha_p(G)$, denoting the maximum size of a set $S$ such that $G[S]$ contains no copy of $K_p$.

For a graph $H$ with $\chi(H)=r\ge3$, the \textbf{decomposition family} $\mathcal{M}(H)$ consists of all bipartite graphs obtained from $H$ by deleting $r-2$ color classes in some proper $r$-coloring of~$H$.  
Following {\L}uczak and Thomass{\'e}~\cite{luczak2010coloring}, a graph $H$ with $\chi(H)=3$ is called \textbf{near-acyclic} if it admits a partition into a forest $F$ and an independent set $S$ such that every odd cycle of $H$ meets $S$ in at least two vertices.\footnote{For instance, $C_5$ is near-acyclic.}  
We say that $H$ is \textbf{$r$-near-acyclic} if $\chi(H)=r\ge3$ and one can delete $r-3$ independent sets from $H$ to obtain a near-acyclic graph.

\subsection{Extremal graphs}\label{sec2.1}

We begin by defining the building blocks for the constructions that show our main theorems are sharp. These constructions rely on graphs that simultaneously have high girth and high chromatic number.

\begin{theorem}[\L uczak and Thomass\'e \cite{luczak2010coloring}]\label{LT}
  For every $k, \ell \in \mathbb{N}$, a real number $\alpha>0$ and $n \ge n_0(k,\ell,\alpha)$, there exists an $n$-vertex $(k,\ell,\alpha)$-Borsuk-Hajnal graph $\mathrm{BH}=\mathrm{BH}(n,k,\ell,\alpha)$ on vertex set $U' \cup X \cup W$ satisfying the following:
  \begin{itemize}
      \item $\chi(\mathrm{BH}[U']) \geq k, \quad   g(\mathrm{BH}[U']) \geq \ell, \quad  \text { and } \quad \delta(\mathrm{BH}) \geq (1/3-\alpha) n$;
      \item $|U'| \le \alpha n, \quad  \text { and } \quad  \forall u\in U'$, $|N(u)\cap X| \ge (1/2-\alpha)|X|$;
      
      \item $|X|=2|W|$, $\forall w\in W$, $N(w)=X$, and $X$ is an independent set;
 
      \item Every subgraph $H \subseteq \mathrm{BH}$ with $|H|<\ell$ and $\chi(H)=3$ is near-acyclic, i.e., $\delta_{\chi}(H)=0$.
  \end{itemize}
\end{theorem}

Roughly speaking, the part $\mathrm{BH}[U']$ is a variant of Borsuk graph with high girth and chromatic number where vertices are uniformly distributed on a high-dimensional sphere and two vertices are adjacent if they are at almost antipodal positions. We refer the reader to \cite{allen2013chromatic} for the details of the construction.

We can now define two families of building blocks.

\begin{construction}\label{cons22}
Let $\mathrm{BH}^-$ be the graph on vertex set $U'\cup X$ obtained from $\mathrm{BH}(n,k,\ell,\alpha)$ by removing the vertex set $W$. 
We define two families of graphs, $\mathrm{BH}_{r, \delta}^{\star}$ and $\mathrm{BH}_{r, \delta}^{\star\star}$, which are obtained from $BH^{-}$ by adding vertex sets $Y=Y_1 \cup \cdots \cup Y_{r-3}$ and $Z$.
The edges are added as follows:
\begin{enumerate}
    \item[\rm 1.] The sets $Y_1, \ldots, Y_{r-3}$ form a complete $(r-3)$-partite graph.
    \item[\rm 2.] A complete join is added between $X$ and $Z$.
    \item[\rm 3.] A complete join is added between $Y$ and $U'\cup X \cup Z$.
\end{enumerate}
We choose the parameter $\alpha$ in \cref{LT} sufficiently small so that $|U'| = o(n)$. The size of $X$ is chosen such that the total number of vertices is $n$, with the sizes of the other sets determined by the following ratios (matching the bounds in \cref{thm:thm1}):
\begin{itemize}
    \item For $\mathrm{BH}_{r, \delta}^{\star}:|Y_i|=\frac{1-\delta}{2 \delta-2(r-3)(1-\delta)}|X|$ and $|Z|=\frac{1-2 \delta+(r-3)(1-\delta)}{2 \delta-2(r-3)(1-\delta)}|X|$.
    \item For $\mathrm{BH}_{r, \delta}^{\star\star}:|Y_i|=\frac{6 \delta-1}{2(\delta(5 r-14)-(r-3)(6 \delta-1))}|X|$ and $|Z|=\frac{(1-2 \delta)(5 r-14)+(r-3)(6 \delta-1)}{2(\delta(5 r-14)-(r-3)(6 \delta-1))}|X|$.
\end{itemize}
\end{construction}

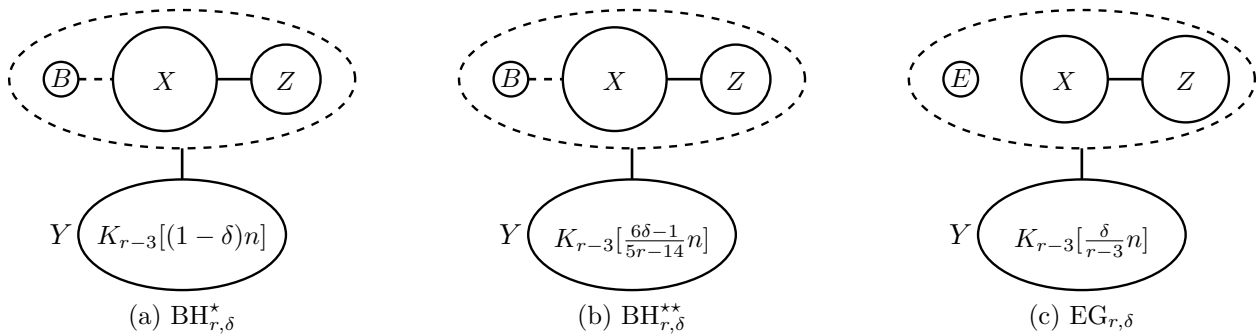
\begin{figure}[!ht]
    \centering
\begin{tikzpicture}[scale=0.46]
  \draw[line width=0.901pt] (0, 3) circle (0.5cm);
  \node[below] at (0,3.6) {\small $B$};
  \draw[line width=0.901pt] (3, 3) circle (1.5cm);
  \node[below] at (3,3.5) {\small $X$};
  \draw[line width=0.901pt] (6.5, 3) circle (1cm);
  \node[below] at (6.5,3.5) {\small $Z$};
  \draw[dashed,line width=1pt] (0.5, 3)  -- (1.5, 3);
  \draw[line width=1pt] (4.5, 3)  -- (5.5, 3);
  \draw[dashed,line width=0.901pt] (3.5, 3) ellipse (5cm and 2cm);
  \draw[line width=0.901pt] (3.5, -1.5) ellipse (3cm and 1.6cm);
  \node[below] at (3.5, -0.9) {\small $K_{r-3}[(1-\delta)n]$};
  \node[below] at (0, -0.9) {$Y$};
  \node[below] at (3.5, -3.2) {\small (a) $\mathrm{BH}^{\star}_{r,\delta}$};
  \draw[line width=1pt] (3.5, 1)  -- (3.5, 0.1);

  \draw[line width=0.901pt] (0+13.001, 3) circle (0.5cm);
  \node[below] at (0+13.001,3.6) {\small $B$};
  \draw[line width=0.901pt] (3+13.001, 3) circle (1.5cm);
  \node[below] at (3+13.001,3.5) {\small $X$};
  \draw[line width=0.901pt] (6.5+13.001, 3) circle (1cm);
  \node[below] at (6.5+13.001,3.5) {\small $Z$};
  \draw[dashed,line width=1pt] (0.5+13.001, 3)  -- (1.5+13.001, 3);
  \draw[line width=1pt] (4.5+13.001, 3)  -- (5.5+13.001, 3);
  \draw[dashed,line width=0.901pt] (3.5+13.001, 3) ellipse (5cm and 2cm);
  \draw[line width=0.901pt] (3.5+13.001, -1.5) ellipse (3cm and 1.6cm);
  \node[below] at (3.5+13.001, -0.9) {\small $K_{r-3}[\frac{6\delta-1}{5r-14}n]$};
  \node[below] at (0+13.001, -0.9) {$Y$};
  \node[below] at (3.5+13.001, -3.2) {\small (b) $\mathrm{BH}^{\star\star}_{r,\delta}$};
  \draw[line width=1pt] (3.5+13.001, 1)  -- (3.5+13.001, 0.1);

  \draw[line width=0.901pt] (0+26.001, 3) circle (0.5cm);
  \node[below] at (0+26.001,3.6) {\small $E$};
  \draw[line width=0.901pt] (3+26.001, 3) circle (1.25cm);
  \node[below] at (3+26.001,3.5) {\small $X$};
  \draw[line width=0.901pt] (6.5+26.001, 3) circle (1.25cm);
  \node[below] at (6.5+26.001,3.5) {\small $Z$};
  \draw[line width=1pt] (4.25+26.001, 3)  -- (5.25+26.001, 3);
  \draw[dashed,line width=0.901pt] (3.5+26.001, 3) ellipse (5cm and 2cm);
  \draw[line width=0.901pt] (3.5+26.001, -1.5) ellipse (3cm and 1.6cm);
  \node[below] at (3.5+26.001, -0.9) {\small $K_{r-3}[\frac{\delta}{r-3}n]$};
  \node[below] at (0+26.001, -0.9) {$Y$};
  \node[below] at (3.5+26.001, -3.2) {\small (c) $\mathrm{EG}_{r,\delta}$};
  \draw[line width=1pt] (3.5+26.001, 1)  -- (3.5+26.001, 0.1);
  
\end{tikzpicture}
\caption{The illustration of $\mathrm{BH}^{\star}_{r,\delta}$, $\mathrm{BH}^{\star\star}_{r,\delta}$ and $\mathrm{EG}_{r,\delta}$, where $B$ denotes Borsuk graph and $E$ denotes Erd\H{o}s graph.}
\label{fig:construction}
\end{figure}

\begin{remark}
Let $H$ be a graph with $\chi(H)=r$ that is not $r$-near-acyclic, which corresponds to the threshold $\delta_\chi(H)=\frac{2r-5}{2r-3}$.
We note that our extremal constructions, $\mathrm{BH}^{\star}_{r,\delta}$ and $\mathrm{BH}^{\star\star}_{r,\delta}$, are derived from the $r$-Borsuk–Hajnal graph by  adjusting the sizes of its independent sets. 
The underlying $r$-Borsuk–Hajnal graph was shown to be $H$-free in~\cite{allen2013chromatic}, providing the structural basis for our construction.
\end{remark}

 For any $k, \ell \in \mathbb{N}$, a \textbf{$(k, \ell)$-Erd\H{o}s graph} is a graph with chromatic number at least $k$ and girth at least $\ell$.
 Erd\H{o}s \cite{erdos1959graph} showed the existence of such graphs for any $k, \ell$.

\begin{construction}\label{cons3}
Let $\mathrm{EG}_{r,\delta}$ be the graph obtained from an Erd\H{o}s graph $E$ by adding vertex sets $X$, $Z$ and $Y=Y_1\cup \cdots\cup Y_{r-3}$.
The edges are added as follows:
\begin{enumerate}
    \item[\rm 1.] The sets $Y_1, \ldots, Y_{r-3}$ form a complete $(r-3)$-partite graph.
    \item[\rm 2.] A complete join is added between $X$ and $Z$.
    \item[\rm 3.] A complete join is added between $V(E)\cup X\cup Z$ and $Y$.
\end{enumerate}
We choose the Erd\H{o}s graph $E$ such that $|V(E)| = o(n)$. The size of each set $Y_i$ is chosen such that the total number of vertices is $n$, with the sizes of $X$ and $Z$ determined by the following ratio:
\begin{itemize}
    \item $|X|=|Z|=\frac{(r-3)(1-\delta)}{2\delta}|Y_i|$.
\end{itemize}
\end{construction}

\begin{remark}
Let $H$ be a graph with $\chi(H)=r$. 
We verify that our construction $\mathrm{EG}_{r,\delta}$ is indeed $H$-free. 
Since the Erd\H{o}s graph component is chosen to have a large girth, any subgraph induced by $|H|$ vertices in $V(E)\cup X\cup Z$ is $2$-colorable. 
Hence, any subgraph induced by $|H|$ vertices in the entire graph $\mathrm{EG}_{r,\delta}$ is $(r-1)$-colorable, which implies that $\mathrm{EG}_{r,\delta}$ is $H$-free.
\end{remark}

\subsection{The Regularity lemma}

In this section, we state some auxiliary theorems. First we introduce the famous regularity lemma. Let $(A, B)$ be a pair of subsets of vertices of $G$. Let $e(A, B)$ denote the number of edges with one endpoint in $A$ and the other in $B$. Define the \textbf{density} of the pair $(A, B)$ as $d(A, B)=\frac{e(A, B)}{|A||B|}$. For any $\varepsilon>0$, we say that $(A, B)$ is \textbf{$\varepsilon$-regular} if $|d(A, B)-d(X, Y)|<\varepsilon$ for every $X \subseteq A$ and $Y \subseteq B$ with $|X| \geq \varepsilon|A|$ and $|Y| \geq \varepsilon|B|$. Moreover, given $0 < d < 1$, we say that $(A, B)$ is \textbf{$(\varepsilon, d)$-regular} if it is $\varepsilon$-regular and has density at least $d$.

A partition $V_0 \cup V_1 \cup \cdots \cup V_k$ of $V(G)$ is said to be an \textbf{$\varepsilon$-regular partition} if $|V_0| \leq \varepsilon n,|V_1|=\cdots=|V_k|$, and all but at most $\varepsilon k^2$ of the pairs $(V_i, V_j)$, where $1\le i,j\le k,$ are $\varepsilon$-regular. Given an $\varepsilon$-regular partition $V_0 \cup V_1 \cup \cdots \cup V_k$ of $V(G)$ and $0<d<1$, we define a graph $R$, called the \textbf{$(\varepsilon, d)$-reduced graph} of $G$, as follows: the vertex set is $V(R)=[k]$ and two vertices $i,j$ satisfy $ij \in E(R)$ if and only if $(V_i, V_j)$ is an $(\varepsilon, d)$-regular pair. The partition classes $V_1, \ldots, V_k$ are called the \textbf{clusters} of $G$. For brevity, for each $I\subseteq [k]$ we will write $V_I=\cup_{i\in I}V_i$.

We will use the following minimum degree form of the regularity lemma.

\begin{theorem}[Regularity Lemma-Degree Form, Theorem~1.10 of \cite{komlos1995szemeredi}]\label{thm:RL}
    Let $0<\varepsilon<d<\delta<1$, and let $k_0 \in \mathbb{N}$. There exists a constant $k_1=k_1(k_0, \varepsilon, \delta, d)$ such that the following holds. Every graph $G$ on $n>k_1$ vertices, with minimum degree $\delta(G) \geq \delta n$, has an $(\varepsilon, d)$-reduced graph $R$ on $k$ vertices, with $k_0 \leq k \leq k_1$ and $\delta(R) \geq(\delta-d-\varepsilon) k$.
\end{theorem}

\begin{lemma}[\cite{allen2013chromatic}]\label{lmm:lmm5.3}
Let $H$ be an $r$-near-acyclic-graph and $\varepsilon,\beta,d>0$ be reals. Let $G$ be a graph, and $\chi(G[U])\ge C$ for some sufficiently large constant $C=C(H,\varepsilon,\beta,d)$. 
Let $U$, $X$ and $Y_1, \ldots, Y_{r-3}$ be pairwise disjoint subsets of $V(G)$, with $|X|=|Y_j|$ for each $j \in[r-3]$. Suppose that $(X, Y_j)$ and $(Y_i, Y_j)$ are $(\varepsilon, d)$-regular for each $i \neq j$, and that
\begin{align*}
|N(u) \cap X| \geq \beta|X| \quad \text{and} \quad |N(u) \cap Y_j| \geq (1/2+\beta)|Y_j|
\end{align*}
for every $u \in U$ and $j \in[r-3]$. Then $H\subseteq G$.
\end{lemma}

The following lemma provides a standard estimate that relates the number of edges in the original graph $G$ to the number of edges in its reduced graph $R$.

\begin{lemma}
For $\varepsilon>0$ and $d \in(0,1)$, let $G$ be a graph on $n$ vertices with an ($\varepsilon, d$)-reduced graph $R$ on $k$ vertices, where $1/k\le 2\varepsilon$. Then 
  $e(G)\le e(R)\frac{n^2}{k^2}+\left(\frac{d}{2} + 2\varepsilon\right)n^2$.
\end{lemma}

\begin{proof}
 The edges of $G$ can be categorized into the following four types:
\begin{itemize}
\item \textbf{Edges within clusters:} The number of edges with both endpoints in $V_i$ is at most $k\binom{n/k}{2}$;

\item \textbf{Edges in irregular pairs:} The number of edges between distinct clusters $V_i$ and $V_j$ such that $(V_i, V_j)$ is not $\varepsilon$-regular is at most $\varepsilon k^2\frac{n^2}{k^2}$;

\item \textbf{Edges in sparse regular pairs:} The number of edges between distinct clusters $V_i$ and $V_j$ where $(V_i,V_j)$ is $\varepsilon$-regular but $d(V_i,V_j)<d$ is at most $d\binom{k}{2}\frac{n^2}{k^2}$;

\item \textbf{Edges in dense regular pairs:} The edges between distinct clusters $V_i$ and $V_j$ such that $(V_i, V_j)$ is $\varepsilon$-regular and $d(V_i, V_j) \ge d$ correspond to the edges of $R$. The number of such edges is at most $e(R) \frac{n^2}{k^2}$.
\end{itemize} 
Summing the upper bounds for these four types yields the desired result.
\end{proof}

\subsection{Zykov symmetrization}

We shall want to apply the \textit{Zykov symmetrization}, defined as follows. 
Given a graph $G$ and two non-adjacent vertices $u, v \in$ $V(G)$, the graph $Z_{u,v}(G)$ is obtained by replacing $u$ with a twin of $v$.
That is, we delete all edges incident to $u$ and add edges between $u$ and the neighbors of $v$ instead. 
Note that the relation of being twins forms an equivalence relation, which partitions the vertex set into \textbf{twin classes}.

Let $G$ be a graph with vertex set $V(G)=[n]$. We define a total order $\prec$ on $[n]$ by letting $i \prec j$ if either $d(i) < d(j)$ or $d(i) = d(j)$ and $i < j$. We refer to $Z_{u,v}$ as an \textbf{increasing Zykov symmetrization} (IZS) if $u \prec v$.

Given a vertex set $A\subseteq V(G)$, we construct $Z(G|A)$ through the following iterative procedure:
\begin{enumerate}
    \item Partition $A$ into twin classes $A(i_1), A(i_2), \ldots, A(i_a)$, indexed by their smallest vertex labels $i_1 < \cdots < i_a$. That is, $i$ is the smallest vertex in $A(i)$.
    \item While some pair of twin classes remains non-adjacent
\begin{enumerate}
\item Among all non-adjacent class pairs $(A(k), A(j))$ with $k<j$, select the pair minimizing $k+j$.
\item If $j \prec k$, merge $A(j)$ into $A(k)$. Otherwise, merge $A(k)$ into $A(j)$.
\item Update labels to maintain the ordering $i_1<i_2<\cdots<i_{a'}$ for the remaining classes. 
\end{enumerate}
    \item The process terminates when every pair of twin classes induces a complete bipartite graph. The resulting graph is $Z(G|A)$.
\end{enumerate}

To ensure the uniqueness of the symmetrization process, we will throughout this chapter assume that any graph $G$ with $n$ vertices has vertex set $[n]$.

Observe that $\omega(Z_{u,v}(G))\le \omega(G-u)$ and $\chi(Z_{u,v}(G))\le \chi(G-u)$. The following proposition follows readily from the construction.

\begin{proposition}\label{prop:zykov}
    Let $G$ be a graph, and let $A \subseteq V(G)$. Denote $Z := Z(G|A)$.
    Then
\begin{itemize}
    \item[\rm (1)] $e(G)\le e(Z)$;
    \item[\rm (2)] $\omega(Z)\le \omega(G)\text{ and }\omega(Z[A])\le \omega(G[A])$;
    \item[\rm (3)] every pair of non-adjacent vertices in $A$ has the same neighborhood.
\end{itemize}
\end{proposition}

We define an $r$-clique $K$ as an \textit{$(X,Y,r)$-clique} if $|V(K)\cap X|=1$ and $|V(K)\cap Y|=r-1$. A graph $G$ is said to be \textbf{$(X,Y,r)$-free} if it contains no $(X,Y,r)$-cliques. The following proposition shows that the $(X,Y,r)$-free property is preserved under Zykov symmetrization.

\begin{proposition}\label{62}
    Let $X$, $Y$ be two disjoint sets of $V(G)$. If $G$ is $(X,Y,r)$-free, then the Zykov symmetrization $Z(G|X)$ and $Z(G|Y)$ both  preserve the $(X,Y,r)$-free property.
\end{proposition}

\begin{proof}
 We prove that a single symmetrization operation on a pair of non-adjacent vertices in $X$ or $Y$ preserves the $(X,Y,r)$-free property. Since each such operation only modifies the graph locally without introducing new $(X,Y,r)$-cliques, it follows by induction that any finite sequence of such operations also preserves the property. It thus suffices to prove the claim for a single symmetrization step.

    Let $x,x'\in X$ be two non-adjacent vertices.
    Assume for contradiction that $Z_{x',x}(G)$ contains an $(X,Y,r)$-clique $K$.
    By definition, this implies $V(K) \cap X = \{x'\}$. Since $x'$ and $x$ have identical neighborhoods, replacing $x'$ with $x$ yields an $(X,Y,r)$-clique in $G - x'$, contradicting the $(X,Y,r)$-freeness of $G$.

    Let $y,y'\in Y$ be two non-adjacent vertices.
    Assume for contradiction that $Z_{y',y}(G)$ contains an $(X,Y,r)$-clique $K$, then $K$ must contain $y'$. In this case, we have $V(K) = \{x, y', y_1, \ldots, y_{r-2}\}$, where $x \in X$ and $y_i \in Y$ for all $i \in [r - 2]$. Since $y$ and $y'$ share the same neighborhood, replacing $y'$ with $y$ yields an $(X,Y,r)$-clique in $G - y'$, again contradicting the assumption.
\end{proof}

\section{Auxiliary Lemmas}

The famous Andr{\'a}sfai-Erd\H{o}s-S{\'o}s theorem is a cornerstone in the study of chromatic properties of $K_r$-free graphs with large minimum degree.

\begin{theorem}[Andr{\'a}sfai, Erd\H{o}s and S{\'o}s \cite{andrasfai1974connection}]\label{thm:1974Andrasfai}
    Let $r\geq 3$ and let $G$ be a $K_r$-free graph on $n$ vertices such that $\delta(G)>\frac{3 r-7}{3 r-4} n$. Then $\chi(G) \leq r-1$.
\end{theorem}

The next two lemmas are generalized Tur\'an-type theorems that provide the core machinery for estimating edge density.

\begin{lemma}\label{lmm:basic}
Let $G$ be a $K_r$-free graph with vertex set $V(G)=A\cup B$, where $|V(G)|=n$ and $|A|=a$. 
If $G[A]$ is $K_t$-free with $|A|\ge \frac{t-1}{r-1}n$, then $e(G)\le e(T_{t-1}(a)\vee T_{r-t}(n-a)).$
\end{lemma}
\begin{proof}
 Let $\phi(a) = e(T_{t-1}(a) \vee T_{r-t}(n-a))$. 
 We first observe that $\phi$ is a decreasing function for $a \ge \frac{t-1}{r-1}n$. 
 We now prove \cref{lmm:basic} by induction on $|B|$. 
 If $|B| = 0$, then $V(G)=A$, and the claim follows directly from  Tur\'an's theorem. 
 Assume the lemma holds for all graphs with $|B| \le k$ where $k\le \frac{r-t}{t-1}n-1$. Let $G$ be a graph with $|B| = k+1$. 
By Proposition~\ref{prop:zykov}, we may use $Z(G|A)$ instead of $G$. Hence, by \cref{prop:zykov}, we may assume that $G[A]$ is an $s$-partite graph with $s \le t-1$, where each partition class has identical neighborhoods in $G$. 
 We analyze two subcases based on the structure of $G[A]$:
 \begin{itemize}
     \item If some vertex $z \in B$ is non-adjacent to a partition class of $G[A]$, move $z$ into $A$, forming $A' = A \cup \{z\}$. Since $G[A']$ remains $K_t$-free, the induction hypothesis implies $e(G) \le \phi(a+1) \le \phi(a)$. 
     \item If $s \le t-2$, then arbitrarily relocate any vertex $z \in B$ to $A$, forming $A' = A \cup \{z\}$. Again, $G[A']$ is $K_t$-free, and by induction $e(G) \le \phi(a+1) \le \phi(a)$. 
 \end{itemize}
 If neither subcase applies, then $s = t-1$ and every vertex in $B$ is adjacent to all vertices in $A$.
 Consequently, $G[B]$ must be $K_{r-t+1}$-free.
 Applying  Tur\'an's theorem to $G[B]$, we obtain:
\[e(G) \le e(G[A]) + e(G[B]) + a(n-a) \le e(T_{t-1}(a) \vee T_{r-t}(n-a)),\]
which completes the induction.
\end{proof}

\begin{lemma}\label{lmm:xyz}
 Let $G$ be a $K_r$-free graph on $n$ vertices. 
 Suppose $V(G)$ has a partition $V(G)=X\cup Y\cup Z$ such that 
\begin{itemize}
    \item $G[Y]$ is $K_{r-2}$-free;
    \item $G[X\cup Y]$ is $K_{r-1}$-free;
    \item $\frac{|Y|}{r-3}\ge |X|\ge |Z|$.
\end{itemize} 
Then
\[e(G)\le \frac{(r-4)}{2(r-3)}|Y|^2+|X||Y|+|Y||Z|+|Z||X|.\]
\end{lemma}
\begin{proof}
When $|X| = |Z|$, applying \cref{lmm:basic} with $Y = A$ and $t = r - 2$ immediately yields the desired bound. Now we assume that $|X|>|Z|$.

Define $f(x,y,z)=\frac{(r-4)}{2(r-3)}y^2+xy+yz+zx$.  A routine calculation shows that, under the assumption $\frac{y}{r - 3} \ge x > z$, we have $f(x, y, z) \ge f(x - 1, y + 1, z)$ and $f(x, y, z) \ge f(x, y + 1, z - 1)$. In the following, we denote $|X| = x$, $|Y| = y$, and $|Z| = z$ for convenience.

 We now prove \cref{lmm:xyz} by induction on $x+z$. 
 If $x=z=0$, then $V(G)=Y$, and the bound follows directly from  Tur\'an's theorem.
 Assume the lemma holds for all graphs with $x+z \le k$ and $x>z$.
 Let $G$ be a graph with $x+z=k+1$.
 Without loss of generality, we may use $Z(G|Y)$ instead of $G$.
 By \cref{prop:zykov}, we may assume that $G[Y]$ is an $s$-partite graph with $s \le r-3$, where each partition class has identical neighborhoods in $G$. 
 We consider two subcases based on the structure of $G$:
\begin{itemize}
     \item If there exists a vertex $u \in X$ that is non-adjacent to some partition class of $G[Y]$, then move $u$ into $Y$,  forming $X'=X\setminus \{u\}$ and $Y' = Y \cup \{u\}$. Since $G[Y']$ remains $K_{r-2}$-free, the induction hypothesis gives $e(G) \le f(x-1,y+1,z)\le f(x,y,z)$. 
     \item If $s \le r-4$, then arbitrarily move a vertex $u$ into $Y$, forming $X'=X\setminus \{u\}$ and $Y' = Y \cup \{u\}$.
     As $s+1\le r-3$, $G[Y']$ remains $K_{r-2}$-free.
     By induction, $e(G) \le f(x-1,y+1,z)\le f(x,y,z)$. 
 \end{itemize}
 If neither subcase applies, then $s=r-3$, $X$ and $Y$ are completely adjacent.
 Since $G[X\cup Y]$ is $K_{r-1}$-free, $X$ must be an independent set.

If there exists $u \in Z$ that is not adjacent to some partition class of $G[Y]$, then move $u$ to one of these parts and let $Y' = Y \cup \{u\}$, $Z'=Z\setminus \{u\}$,  ensuring that $Y'$ remains $K_{r-2}$-free. Since $X$ is an independent set, $X\cup Y'$ is $(r-2)$-partite, ensuring that $X\cup Y'$ remains $K_{r-1}$-free. Hence, by the induction hypothesis, we have $e(G) \le f(x,y+1,z-1)\le f(x,y,z)$.  

Otherwise, all vertices in $X\cup Z$ are adjacent to all vertices in $Y$, hence $X\cup Z$ must be $K_3$-free. 
Applying \cref{lmm:basic} with $A=X$ and $t=2$, we have $e(X\cup Z)\le xz$. By Tur\'an's theorem, we have $e(Y)\le \frac{(r-4)}{2(r-3)}y^2$.
Therefore,  we have
\[e(G)\le e(Y)+e(X\cup Z)+(x+z)y\le  \frac{(r-4)}{2(r-3)}y^2+xy+yz+zx.\qedhere\]\end{proof}

In order to apply the preceding edge-bounding theorems, we must first analyze the structure of the graph—particularly the reduced graph yielded by the regularity lemma. The following set of lemmas provides the technical tools for this purpose. 
Recall that a \textbf{matching} in a graph $G$ is a set of edges without shared vertices, and
the \textbf{matching number} of $G$ is the size of a largest matching in $G$. 

\begin{lemma}[Theorem 1.3.1 in \cite{lovasz2009matching}]\label{lmm:Hall-cond}
    Suppose that $B(X,Y)$ is a bipartite graph with vertex set $X\cup Y$.
    Let $t:=\max\{|S|-|N(S)|:S\subseteq X\}$.
    Then the matching number of $B$ is $|X|-t$.
\end{lemma}

\begin{lemma}[see Lemma 9 of \cite{allen2013chromatic}\label{lmm:lmm9}]
    Let $\beta, \delta>0$ and $r, t \in \mathbb{N}$, let $F$ be a forest, and suppose that $H \subseteq F \vee K_{r-2}[t]$. Let $G$ be a graph on $n$ vertices, and $T \subseteq V(G)$.
\begin{enumerate}
    \item[\rm (a)] If $\delta(G) \geq \delta n$ and $ |T| \geq\big(\beta^\frac{1}{r-2} (r-2)+(1-\delta)(r-3)\big) n$, then $G[T]$ contains at least $\beta n^{r-2}$ copies of $K_{r-2}$.
    \item[\rm (b)] If $G[N(x)]$ contains at least $\beta n^{(r-1)|H|}$ copies of $K_{r-1}[|H|]$ for every $x \in T$, then we have either $H \subseteq G$ or $|T| \leq |H| / \beta$.
\end{enumerate}
\end{lemma}

\begin{lemma}[see Lemma 10 of \cite{allen2013chromatic}]\label{lmm:lmm10}
    Let $\beta, \delta>0$ and $r, t \in \mathbb{N}$, let $F$ be a forest, and suppose that $H \subseteq F \vee K_{r-2}[t]$. 
    Let $G$ be an $H$-free graph on $n$ vertices, and let $T \subseteq V(G)$ be such that every edge $x y \in E(G[T])$ is contained in at least $\beta n^{r-2}$ copies of $K_{r}$ in $G$.
    Then 
    $\chi(G[T]) \leq(2|F| / \beta')+1$ for some constant $\beta'=\beta'(\beta,r,t)$
\end{lemma}

\section{Proofs of main results}

\subsection{Proof of Theorem \ref{thm:thm1}}

We are now ready to solve the case $\delta_\chi(H)=\frac{2r-5}{2r-3}$. 

  Fix  $1\gg\varepsilon\gg \beta\gg d\gg \varepsilon_0>0$ and a large constant $C=C(H,\varepsilon)$.
 Let $G$ be an $H$-free graph with $\delta(G)\ge (\delta+2\varepsilon) n$.
 Applying the Regularity Lemma to $G$ yields
\begin{itemize}
    \item A partition $V(G)=V_0\cup V_1\cup \cdots \cup V_k$, where $|V_0| \leq \varepsilon_0 n$ and $|V_i|=\frac{n-|V_0|}{k}$ for $i \geq 1$.
    \item An $(\varepsilon_0,d)$-reduced graph $R$ on vertex set $[k]$, with $\delta(R)\geq \big(\delta+\varepsilon\big)k$. 
\end{itemize}

 We now define a refined partition of $V(G)$. For each pair of index sets $Y \subseteq I \subseteq [k]$, let
\begin{align*}
X(I, Y) := \{v \in V(G) :\; 
& i \in I \Leftrightarrow \left|N(v) \cap V_i\right| \geq \beta \left|V_i\right| \\
& \text{ and } i \in Y \Leftrightarrow \left|N(v) \cap V_i\right| \geq \left( 1/2 + \beta \right) \left|V_i\right| \}.
\end{align*}

We will also need the following lemma, which provides a structural consequence for $H$-free graphs.

\begin{lemma}[see Claim 3.3 of \cite{liu-shangguan-wu-xue}]\label{lmm:0-lmm-1/2-lmm}
   Let $F$ be a forest, and suppose that $H \subseteq F+K_{r-2}(t)$. 
   Let $G$ be an $H$-free graph on $n$ vertices and $R$ be the $(\varepsilon,d)$-reduced graph of $G$.
   For $X(I,Y)\neq\varnothing$, there are constants $C_1=C_1(H,\beta)$ and $C_2=C_2(H,\beta)$ such that the following hold:
  \begin{itemize}
      \item [{\rm (i)}] If $R[I_1]$ contains a copy of $K_{r-1}$, then $\chi(G[X(I,Y)])\le C_1$;
      \item [{\rm (ii)}] If $R[I_2]$ contains a copy of $K_{r-2}$, then $\chi(G[X(I,Y)])\le C_2$.
  \end{itemize}
\end{lemma}

\begin{proof}[\textbf{Proof of \cref{thm:thm1}}]
Choose $I,Y$ such that $\chi(G[X(I, Y)])$ is maximized. Then $\chi(G[X(I, Y)]) > C / 4^k$.
 By \cref{lmm:0-lmm-1/2-lmm}, $R[I]$ is $K_{r-1}$-free, and $R[Y]$ is $K_{r-2}$-free. 
  
Now, let $X=I\setminus Y$, and $Z=[k]\setminus I$. Consider a vertex $x\in X(I, Y)$. By the minimum degree condition, we have
\begin{align*}
    (\delta+2\varepsilon) n\le d(x)\le |X|(1/2+d)\frac{n}{k}+|Y|\frac{n}{k}.
\end{align*}
  This implies that $|X|/2+|Y|>\delta k$.

  By the definition of $X(I, Y)$, each vertex $x\in X(I, Y)$ has at most $(\varepsilon_0+d)n$ neighbors outside $V_I$.
  Therefore, $|N(x)\cap V_I|\ge (\delta+2\varepsilon)n-(\varepsilon_0+d)n>(\delta+\varepsilon)n$ as $\varepsilon\gg d\gg \varepsilon_0$.
  We claim that $|I|> ((r-1)\delta-r+3)k$.
  Otherwise, for any edge $xy\in G[X(I,Y)]$, we have 
 $$|N(x)\cap N(y)\cap V_I|\ge 2(\delta+\varepsilon )n-|V_I|\ge (2\varepsilon+(1-\delta)(r-3))n.$$
 Hence, applying \cref{lmm:lmm9} (a) with $\beta=(\frac{2\varepsilon}{r-2})^{r-2}$ and $T=N(x)\cap N(y)$, we deduce that $G[N(x)\cap N(y)]$ contains at least $(\frac{2\varepsilon}{r-2})^{r-2}n^{r-2}$ copies of $K_{r-2}$. By \cref{lmm:lmm10}, this implies that $\chi(G[X(I, Y)])$ is bounded, contradicting our choice of $I,Y$.

We proceed by analyzing the following three cases:

\begin{mycase}{Case 1}
    $|Y|<(r-3)|X|$.
\end{mycase}

In this case, we have $\delta k<|X|/2+|Y|<\frac{2r-5}{2}|X|$, which implies that $|X|>\frac{2\delta }{2r-5}k$.
  Thus
\begin{align*}
    |I|=|X|+|Y|=|X|/2+(|X|/2+|Y|)>\frac{2r-4}{2r-5}\delta k>\frac{r-2}{r-1}k.
\end{align*}

  Since $R$ is $K_r$-free and $R[I]$ is $K_{r-1}$-free, by \cref{lmm:basic}, we have
\begin{align*}
  e(R)\le \frac{r-3}{2(r-2)}|I|^2+|I|(k-|I|)\le \frac{2\delta(r-2)(\delta-\delta r+2r-5)}{(2r-5)^2}k^2.
\end{align*}
Therefore, it suffices to prove that $\frac{2\delta(r-2)(\delta-\delta r+2r-5)}{(2r-5)^2}\le f_1(r,\delta)$. 

When  $\frac{5r-13}{5r-8}<\delta\le\frac{2r-5}{2r-3}$, it suffices to show that 
\[4\delta(r-2)(\delta-\delta r+2r-5)\le (2r-5)^2[(-5r^2+17r-14)\delta^2+(10r^2-44r+46)\delta-5r^2+27r-36]\]
which reduces to
\[[(2r-3)\delta-(2r-5)][(10r^2-39r+38)\delta-(10r^2-49r+60)]\le 0.\]
Since $(2r-3)\delta-(2r-5)\le 0$, it remains to verify that 
\[(10r^2-39r+38)\frac{5r-13}{5r-8}-(10r^2-49r+60)\ge 0 \Longleftrightarrow  5r-14\ge 0.\]

When  $\frac{2r-5}{2r-2}\le\delta\le\frac{5r-13}{5r-8},$  it suffices to show that
\[ 4\delta(r-2)(\delta-\delta r+2r-5)(5r-14)\le (2r-5)^2(-4 (r-1)\delta^2+4 (2 r-5)\delta+r-3)\]
which is equivalent to $((2r-2)\delta-(2r-5))^2\ge 0$.

\begin{mycase}{Case 2}
    $|Y|\ge(r-3)|X|$ and $|X|< |Z|$.
\end{mycase}

In this case, $k=(|X|/2+|Y|)+|X|/2+|Z|\ge \delta k+3|X|/2$, which implies that $|X|<\frac{2(1-\delta)k}{3}$. Therefore, we have \[|Y|\ge \delta k-|X|/2\ge  \frac{(4\delta-1)k}{3}>\frac{r-3}{r-1}k.\]
 Since $R[Y]$ is $K_{r-2}$-free,  applying \cref{lmm:basic} yields
\[e(R)\le \binom{r-3}{2}\left(\frac{|Y|}{r-3}\right)^2+|Y|(k-|Y|)+\frac{(k-|Y|)^2}{4}<\frac{(8-8r)\delta^2+(16r-40) \delta+r-4}{18(r-3)}k^2.\]

Hence, we only need to show that $\frac{(8-8r)\delta^2+(16r-40) \delta+r-4}{18(r-3)}\le f_1(r,\delta)$. 

Note that $\frac{\delta(R[Y])}{|Y|} \ge \frac{\delta k-(k-|Y|)}{|Y|}\ge \frac{7\delta-4}{4\delta-1}$.
On the other hand, by  Tur\'an's theorem, $\frac{\delta(R[Y])}{|Y|}\le \frac{r-4}{r-3}$.
Combining the above inequalities yields $\delta\le \frac{3r-8}{3r-5}<\frac{5r-13}{5r-8}$. 
Thus,  it suffices to show that
\[ [(8-8r)\delta^2+(16r-40) \delta+r-4](5r-14)\le 9(r-3)(-4 (r-1)\delta^2+4 (2 r-5)\delta+r-3),\]
which is equivalent to $((2 r -2)\delta-(2 r-5))^2\ge 0$ and thus clearly holds.

\begin{mycase}{Case 3}
    $|Y|\ge(r-3)|X|$ and $|X|\ge |Z|$.
\end{mycase}
In this case, applying \cref{lmm:xyz} yields \[e(R)\le  \frac{(r-4)}{2(r-3)}|Y|^2+|X||Y|+|Y||Z|+|Z||X|.\]
Hence, letting $x=|X|/n, y=|Y|/n$, we only need to show the following claim.
\begin{claim}\label{prop:Lagrangian}
  Let $\frac{2r-5}{2r-2}\le \delta\le \frac{2r-5}{2r-3}$.
  Define $g(x,y)=x(1-x)+y(1-x-y)+\frac{(r-4)}{2(r-3)}y^2$, where $x,y\in[0,1]$ satisfy $y\le (r-3)(1-\delta)$ and $x/2+y\ge \delta$. Then 
  \[g(x,y)\le f_1(r,\delta).\]
\end{claim}
\begin{poc}
 We analyze two intervals of $\delta$ separately.

\begin{mycase}{First interval}
   $\frac{2r-5}{2r-2}\le \delta\le \frac{5r-13}{5r-8}$.
\end{mycase}

To maximize $g(x,y)$ under the constraint $x/2+y\ge \delta$, define the Lagrangian function: $$L(x, y, \lambda)=x(1-x)+y(1-x-y)+\frac{(r-4)}{2(r-3)}y^2+\lambda(x / 2+y-\delta).$$
Then, we take the partial derivatives of $L$ with respect to $x, y$, and $\lambda$, and set them to zero:
\begin{align*}
& \frac{\partial L}{\partial x}=1-2 x-y+\lambda / 2=0 \\
& \frac{\partial L}{\partial y}=1-x-\frac{r-2}{r-3}y+\lambda=0 \\
& \frac{\partial L}{\partial \lambda}=x / 2+y-\delta=0
\end{align*}
Solving these equations simultaneously gives $x=\frac{2(4\delta-r\delta+r-3)}{5r-14}$ and $y=\frac{(6\delta-1)(r-3)}{5r-14}$. That is, $$g(x,y)\le \frac{-4 (r-1)\delta^2+4 (2 r-5)\delta+r-3}{10r-28}=f_1(r,\delta).$$

\begin{mycase}{Second interval}
   $\frac{5r-13}{5r-8}\le \delta\le \frac{2r-5}{2r-3}$.
\end{mycase}

 First, we compute the partial derivative of $g$ with respect to $x$: $\frac{\partial g}{\partial x}=1-2 x-y$.
  From the constraint $x/2+y\ge \delta$, we deduce $x\ge 2\delta-2y$. Substitute this into $\frac{\partial g}{\partial x}$:
\begin{align*}
    \frac{\partial g}{\partial x}=1-2 x-y \le 1-2(2\delta-2y)-y=1-4\delta +3y.
\end{align*}
  Since $y\le (r-3)(1-\delta)$, we obtain 
\begin{align*}
    1-4\delta +3y\le 1-4\delta +3(r-3)(1-\delta)=3r-8-(3r-5)\delta<0
\end{align*}
  when $\delta\ge \frac{5r-13}{5r-8}> \frac{3r-8}{3r-5}$.
Therefore, $\frac{\partial g}{\partial x} < 0$, which implies that $g(x, y)$ is decreasing in $x$. We may thus set $x = 2\delta - 2y$ and define the function
\begin{align*}
    h(y):=g((2\delta-2y),y)=(2\delta-2y)(1-(2\delta-2y))+y(1-(2\delta-2y)-y)+\frac{(r-4)}{2(r-3)}y^2.
\end{align*}

We now simplify $h(y)$ and compute its derivative:
\begin{align*}
    h'(y)=6\delta-1-\frac{5r-14}{r-3}y.
\end{align*}
The condition $h'(y) \ge 0$ leads to
\begin{align*}
    6\delta-1-\frac{5r-14}{r-3}y\ge 0\Longrightarrow y\le \frac{(6\delta-1)(r-3)}{5r-14}.
\end{align*}
  Since $y\le (r-3)(1-\delta)\le \frac{(6\delta-1)(r-3)}{5r-14}$ for all $\frac{5r-13}{5r-8}\le \delta\le \frac{2r-5}{2r-3}$, we have 
\[g(x,y)\le h(y)\le h((r-3)(1-\delta))=f_1(r,\delta).\qedhere\]
\end{poc}
This concludes the proof of \cref{thm:thm1}.
\end{proof}

\subsection{Proof of Theorem \ref{thm:thm2}}

In this subsection, we consider the case $H$ is $r$-near-acyclic.

\begin{proof}[\textbf{Proof of \cref{thm:thm2}}]
Fix  $1\gg\varepsilon\gg \beta\gg d\gg \varepsilon_0>0$ and a large constant $C=C(H,\varepsilon)$.
 Let $G$ be an $H$-free graph with $\delta(G)\ge (\delta+2\varepsilon) n$.
Applying the Regularity Lemma to $G$ yields
\begin{itemize}
    \item A partition $V(G)=V_0\cup V_1\cup \cdots \cup V_k$, where $|V_0| \leq \varepsilon_0 n$ and $|V_i|=\frac{n-|V_0|}{k}$ for $i \geq 1$.
    \item An $(\varepsilon_0,d)$-reduced graph $R$ on vertex set $[k]$, with $\delta(R)\geq \big(\delta+\varepsilon\big)k$. 
\end{itemize}

For each pair of index sets $Y \subseteq I \subseteq [k]$, let
\begin{align*}
X(I, Y) := \{v \in V(G) :\; 
& i \in I \Leftrightarrow \left|N(v) \cap V_i\right| \geq \beta \left|V_i\right| \\
& \text{ and } i \in Y \Leftrightarrow \left|N(v) \cap V_i\right| \geq \left( 1/2 + \beta \right) \left|V_i\right| \}.
\end{align*}

 Choose $I,Y$ such that $\chi(G[X(I, Y)])$ is maximized. Then $\chi(G[X(I, Y)]) > C / 4^k$.  Let $X=I\setminus Y$. Arguing as in the proof of Theorem~\ref{thm:thm1}, we obtain $|X|/2 + |Y| > \delta k$.
  By \cref{lmm:0-lmm-1/2-lmm} and \cref{lmm:lmm5.3}, the reduced graph $R$ satisfies:
\begin{itemize}
    \item $R[I]$ is $K_{r-1}$-free.
    \item $R[Y]$ is $K_{r-2}$-free.
    \item $R$ is $(X,Y,r-2)$-free.
\end{itemize}

 We apply Zykov symmetrization on $X$ and $Y$ successively to obtain $R'$, then $e(R)\le e(R')$.
 By \cref{prop:zykov}, $R'[X]$ becomes a complete $p$-partite graph with $X=X_1\cup \cdots\cup X_p$ and $p\le r-2$.
 Similarly,  $R'[Y]$ is a complete $q$-partite graph with $Y=Y_1\cup \cdots\cup Y_q$ and $q\le r-3$. By \cref{prop:zykov} and \cref{62}, $R'[I]$ is $K_{r-1}$-free, $R'[Y]$ is $K_{r-2}$-free, and $R'$ is $(X,Y,r-2)$-free.

\begin{claim}\label{cl:r-3-color}
    There is an $(r-3)$-colorable induced subgraph of $R'$ of size at least $|X|/2+|Y|$.
\end{claim}

\begin{proof}
  To analyze the interaction between $X$ and $Y$, we construct an auxiliary bipartite graph $B:=B(X^*,Y^*)$ defined as follows: 
  Let $X^*=\{x_1,\ldots,x_p\}$ and $Y^*=\{y_1,\ldots,y_q\}$, where each $x_i$ (resp. $y_j$) represents a twin class in $X$ (resp. $Y$).
  An edge $x_iy_j$ exists in $B$ if and only if no edges connect $X_i$ and $Y_j$ in $R'$.

 For any subsets $S\subseteq Y^*$ and $T\subseteq X^*$ with $|S|=s$, $|T|=t$ and $s+t\ge r-1$, the $K_{r-1}$-freeness of $R'$ implies $E_B(S,T)\neq \varnothing$.
 We now show that $|N(S)|\ge p-r+2+s$.
 Otherwise, we have $|X^*\backslash N(S)|\ge p-(p-r+2+s-1)=r-1-s$ and $E_B(S,X^*\backslash N(S))=\varnothing$, a contradiction.
 By \cref{lmm:Hall-cond}, the matching number of $B$ is $|Y^*|-\max\{|S|-|N(S)|\}\ge p+q-r+2$.

 If $X^*\subseteq V(M)$, then $\chi(R'[X\cup Y])=\chi(R'[Y])=q\le r-3$, $R'[X\cup Y]$ is the desired subgraph. Otherwise, choose  a vertex $x_{\ell}\in X^*\setminus V(M)$ and set $X'= X\setminus X_{\ell}$.

 Note that $X_i\cup Y_j$ is an independent set if $x_i$ is adjacent to $y_j$ in $B$.
 We have $\chi(R'[X'\cup Y])\le (p-1)+q-|M|\le r-3$. 

\begin{itemize}
    \item If $|X'|\ge |X|/2$, then $R'[X'\cup Y]$ is an $(r-3)$-colorable induced subgraph of size at least $|X|/2+|Y|$.
    \item If $|X'|< |X|/2$ and $\chi(R'[Y])\le r-3$, since there is no $(X_{\ell},Y,r-2)$-clique, $R'[X_{\ell}\cup Y]$ is an $(r-3)$-colorable induced subgraph of size at least $|X|/2+|Y|$.
\end{itemize}
We complete the proof of the claim.
\end{proof}

By \cref{cl:r-3-color}, there is a $K_{r-2}$-free induced subgraph of $R'$ of size at least $|X|/2+|Y|\ge \delta k\ge \frac{r-3}{r-1}k$.
By \cref{lmm:basic}, $e(R)\le e(R')\le e(T_{r-3}(\delta k)\vee T_{2}(k-\delta k))=f_2(r,\delta)k^2$.
\end{proof}

\section*{Acknowledgment}

This work was completed while the authors were visiting the Institute for Basic Science (IBS).
We are grateful to Hong Liu for helpful suggestions and for the warm hospitality provided during our visit.

\bibliographystyle{abbrv}
\bibliography{references.bib}

\end{document}